\documentclass[12pt]{amsart}

\usepackage{cancel}
\usepackage{latexsym, amssymb, amsmath}
\usepackage{soul,esint}
\usepackage{amsfonts, graphicx}
\usepackage{graphicx,color}
\newcommand{\be}{\begin{equation}}
\newcommand{\ee}{\end{equation}}
\newcommand{\beq}{\begin{eqnarray}}
\newcommand{\eeq}{\end{eqnarray}}

\newtheorem{claim}{Claim}[section]

\newtheorem{thm}{Theorem}[section]

\newtheorem{lma}{Lemma}[section]
\newtheorem{prop}{Proposition}[section]
\newtheorem{cor}{Corollary}[section]
\newtheorem{defn}{Definition}[section]

\theoremstyle{remark}
\newtheorem{rem}{Remark}[section]
\numberwithin{equation}{section}

\def\be{\begin{equation}}
\def\ee{\end{equation}}
\def\bee{\begin{equation*}}
\def\eee{\end{equation*}}

\def\K{K\"ahler }

\def\Ric{\text{\rm Ric}}
\def\Rm{\text{\rm Rm}}

\def\tr{\operatorname{tr}}

\def\e{\varepsilon}

\begin{document}
\title[]
{Deformation of Hermitian metrics}

 \author{Man-Chun Lee}
\address[Man-Chun Lee]{Department of Mathematics, Northwestern University, 2033
Sheridan Road, Evanston, IL 60208}
\email{mclee@math.northwestern.edu}

\author{Ka-Fai Li
}
\address[Ka-Fai Li]{Morningside Center of Mathematics, Academy of Mathematics and Systems Science, Chinese Academy of Sciences, Beijing, 100190, P.R. China}
 \email{kfli@amss.ac.cn}

\renewcommand{\subjclassname}{
  \textup{2010} Mathematics Subject Classification}

\date{\today}

\begin{abstract} 
In this work, we study the deformation of Hermitian metrics with Chern connection. By adapting the conformal perturbation method of Aubin and Ehrlich
to Hermitian setting, we prove that Hermitian metrics with quasi-positive (resp. quasi-negative) second Chern-Ricci curvature can be deformed to one with positive (resp. negative) curvature. 
\end{abstract}


\maketitle

\markboth{}{Deformation of Hermitian metrics}
\section{introduction}

In differential geometry, it is  important to determine those differentiable manifolds which admit metrics of strictly positive or negative curvature. The existence of metrics with curvature of definite sign will often impose strong restriction on the underlying manifolds. In complex geometry, the holomorphic structure are often characterized by various positivity notions in
complex differential geometry and algebraic geometry. For instance, thanks to the Yau's solution \cite{Yau1978} to the Calabi conjecture, it is now known that the existence of \K metric with positive Ricci curvature on $M$ is equivalent to $M$ being Fano. 

On the other hand, the holomorphic sectional curvature also carry much information of the holomorphic structure.  Indeed, thanks to the recent breakthrough by Wu-Yau \cite{WuYau2016}, it is now known that \K manifolds with negative or quasi-negative holomorphic sectional curvature are projective and have ample canonical line bundle \cite{DiverioTrapani2016,TosattiYang2017,WuYau2016-2}. This settles down a long-standing conjecture of Yau
affirmatively. For more related recent works, we refer interested readers to  \cite{HeierLuWong2010,HeierLuWong2016,
HeierLuWongZheng2017,LeeStreets2019,
YangZheng2016}. On the positive side, it is also conjectured by Yau \cite{Yau1982} that compact \K manifolds with
positive holomorphic sectional curvature must be projective and rationally connected. This was solved recently by Yang \cite{Yang2018} where he introduced the concept of RC-positivity for abstract vector bundles. Many interesting properties and applications using the idea of RC-positivity had also been studied and explored, see \cite{Matsumura2018-1,Matsumura2018-2,
Matsumura2018-3,Yang2018,Yang2018-2,
Yang2018-3,Yang2018-4}.  Along the same spirit, Ni-Zheng also introduced \cite{NiZheng2018,NiZheng2018-2} various notions of Ricci curvature and scalar curvature to obtain rational connectedness of compact \K manifolds.

In Hermitian geometry, it is natural to consider the Chern connection of Hermitian metrics which is the connection compatible with both the metric and complex structure. When the Hermitian metric $g$ is non-K\"ahler, there are four notions of Chern-Ricci curvature associated to the Chern connection due to the presence of torsion. The first and second Chern-Ricci curvature are particularly important to the geometric structure. More precisely, the first Ricci curvature $\Ric(g)$ represents the first Chern class of the canonical line bundle while the second Ricci curvature $S(g)$ is elliptic with respect to the Hermitian metric $g$ and deeply related to its differential geometric structure.


We are interested to understand the second Chern-Ricci curvature $S$ of a Hermitian metric. As an analog of the Calabi-Yau theorem, Yang \cite{Yang2018} proved that if a compact \K manifold $M$ admits a smooth Hermitian metric with positive second Chern-Ricci curvature $S$, then $M$ is projective and rationally connected. This generalized the earlier works of Campana \cite{Campana1992} and Koll\'ar-Miyaoka-Mori \cite{KMM1992} which states that Fano manifolds are rationally connected. Recently in \cite{Yang2020}, the result was generalized by weakening positivity to quasi-positivity. In  \cite{Yang2018-3}, Yang conjectured that a projective and rationally connected manifold $M$ must admit a Hermitian metric with positive second Chern-Ricci curvature.

In this paper, motivated by the above mentioned works and conjecture, we study the deformation of Hermitian metrics. In Riemannian geometry, the deformation on metric with quasi-positive Ricci curvature was studied by
 Aubin \cite{Aubin1970} and Ehrlich \cite{Ehrlich1976}. By adapting the conformal perturbation method in their works to Hermitian setting, we have the the deformation on Hermitian metric with quasi-positive second Chern-Ricci curvature. More generally, we have the following.
\begin{thm}\label{deform-Rm}
Let $M$ is a compact complex manifold and $g_0,\tilde g$ are two smooth Hermitian metric (possibly different) on $M$ such that $\tr_{\tilde g}R^{(TM,g_0)}$ is quasi-positive (resp. quasi-negative). Then there is another Hermitian metric $g_1$ conformal to $g_0$ such that $\tr_{\tilde g}R^{(TM,g_1)}$ is positive (resp. negative) on $M$.
\end{thm}

For the definition of $\tr_g R^{(TM,h)}$ and quasi-positivity, we refer to section \ref{background}.
As a special case, we have the deformation result for the second Chern-Ricci curvature $S$. This also provides a supporting evidence to Yang's conjecture on the existence of metric with positive second Chern-Ricci curvature. 
\begin{cor}\label{deform-S}
Suppose $(M,g)$ is a compact Hermitian manifold with quasi-positive(resp. quasi-negative) $S(g)$. Then there is another Hermitian metric $\tilde g$ such that $S(\tilde g)>0$ (resp. $<0$).
\end{cor}




Together with \cite[Corollary 3.7]{Yang2018}, this implies the following RC-positivity of vector bundles under quasi-positive condition.
\begin{cor}
Let $M^n$ be a compact complex manifold with complex dimension $n$. Suppose there are two Hermitian metric $g$ and $\tilde g$ (possibly different) such that $\tr_{\tilde g}R^{(TM,g_0)}$ is quasi-positive, then there is  Hermitian metric $h$ so that 
\begin{enumerate}
\item $(TM^{\otimes k},h^{\otimes k})$ is RC-positive for every $k\geq 1$;
\item $(\Lambda^p TM,\Lambda^p h)$ is RC-positive for every $1\leq p\leq n$. 
\end{enumerate}
\end{cor}

As an corollary of RC-positivity, we have the following implication on complex structure when $M$ is a K\"ahler using \cite[Theorem 1.4]{Yang2018}, giving an alternative proof to \cite[Theorem 1.1]{Yang2020}. 
\begin{cor}
Let $M$ is a compact \K manifold and $g_0,\tilde g$ are two smooth Hermitian metric on $M$ such that $\tr_{\tilde g}R^{(TM,g_0)}$ is quasi-positive, then $M$ is projective and rationally connected. In particular, $M$ is simply connected.
\end{cor}

The analogous result for the first Chern-Ricci curvature $\Ric$ is not true in general. A counter-example was pointed out by Ehrlich in \cite[Theorem 4]{Ehrlich1976-2} where the Hirzebruch surface $\Sigma_2$ is non-Fano but supports a \K metric with quasi-positive $\Ric$ by the construction of Yau \cite{Yau1974}. This demonstrated a fundamental difference between the first and second Chern-Ricci curvature in the non-\K setting. With a stronger curvature assumption, the deformation of the first Ricci curvature is possible. In \cite{Lee2019}, the first author showed that a Hermitian metric $g$ with quasi-negative $\Ric$ can be deformed to one with negative $\Ric$ if in addition $g$ has Griffths non-positive Chern curvature by using a choice of Hermitian geometric flow introduced by Streets-Tian \cite{StreetsTian2011}. Similar result concerning Griffths non-negativity of Chern curvature was proved by Ustinovskiy \cite{Ustinovskiy2019} which generalized the earlier works by Mok \cite{Mok} and Bando \cite{Bando} in the \K case. Indeed, deformation using parabolic flows was found to be powerful to study the underlying manifold. For related works in Hermitian geometry, we refer interested readers to \cite{BedulliVezzoni2017,CalamaiZhou2020,
Gill2011,
LiYuanZhang,PhongPicardZhang2018,
PhongPicardZhang2016,
PhongPicardZhang2016-2,
PhongPicardZhang2018-2,StreetsTian2010,Streets,
TosattiWeinkove2015} and the references therein. The curvature often tends to become positive along the geometric flow. In contrast, the elliptic method employed here seems to be more flexible in the negative case.


The paper is organized as follows: In section 2, we will collect the formulas about the Chern connection. In section 3, we will derive the variational formula for the first, second Chern-Ricci curvature. In section 4, we will consider the deformation inside the injectivity radius. In section 5, we will prove the global deformation theorems.

{\it Acknowledgement}: The authors would like to thank Professor Xiaokui Yang for suggesting the problem and stimulating discussions. The first author was partially supported by NSF grant 1709894.

\section{Chern connection}\label{background}
In this section, we collect some formulas for the Chern connection to avoid notational inconsistency. Let $(M,g)$ be a Hermitian manifold, the {\it Chern connection} $\nabla$ of $g$ is the connection such that $\nabla g=\nabla J=0$ and the torsion has no $(1,1)$ component. In local holomorphic coordinates $\{z^i\}$, the coefficients $\Gamma$ of $\nabla$ is given by
$$\Gamma_{ij}^k=g^{k\bar l}\partial_i g_{j\bar l}.$$
The {\it torsion} of $g$ is given by $T_{ij}^k=\Gamma_{ij}^k-\Gamma_{ji}^k$. We say that $g$ is \K if the torsion vanishes, $T\equiv 0$. 

The \textit{Chern curvature tensor} of $g$ is defined by $$R_{i\bar jk}\,^l=-\partial_{\bar j}\Gamma_{ik}^l.$$
We raise and lower indices by using metric $g$. The first Chern-Ricci curvature is defined by $$R_{i\bar j}=g^{k\bar l}R_{i\bar j k\bar l}=-\partial_i \partial_{\bar j}\log \det g$$ while the second Ricci curvature is defined by $$S_{i\bar j}=g^{k\bar l}R_{ k\bar li\bar j}.$$
Note that if $g$ is not K\"ahler, then $R_{i\bar j}$ is not necessarily equal to $S_{i\bar j}$. And the Chern-scalar curvature $R$ is defined to be $R=g^{i\bar j}R_{i\bar j}=g^{k\bar l}S_{k\bar l}=g^{i\bar j}g^{k\bar l}R_{i\bar j k\bar l}$. In the non-\K setting, there are two more notions of Chern-Ricci curvature associated to the Chern connection but they are not elliptic with respect to $g$ in general. In this note, we will only focus on the first and second Chern-Ricci curvature. In general, we can also trace the tangent bundle component using a different metric. For notational convenience, we denote $$\tr_{\tilde g}R^{(TM,g)}=\tilde g^{k\bar l}R_{k\bar l i\bar j}.$$ 
In this note, all Hermitian metrics are smooth.

\begin{defn}
Let $M$ be a complex manifold and $A_{i\bar j}$ is a tensor on $M$. We say that $A_{i\bar j}$ is quasi-positive (resp. quasi-negative) if $(A_{i\bar j})$ is non-negative (resp. non-positive) everywhere and positive (resp. negative) at some point on $M$. 
\end{defn}

\section{variation formulas for the Chern curvature}
In this section, we will collect some variational formula for the Chern connection along variation of Hermitian metrics.
\begin{lma}\label{evo-Rm}
Suppose $g(t)$ is a family of Hermitian metrics such that $\partial_t g_{i\bar j}=-h_{i\bar j}$, then we have 
\begin{equation*}
\partial_t R_{k\bar li\bar j}=\frac{1}{2}\left(\nabla_k\nabla_{\bar l}+\nabla_{\bar l}\nabla_k \right) h_{i\bar j}-\frac{1}{2} \left(R_{k\bar li}\,^r h_{r\bar j}+R_{k\bar l}\,^{\bar s}_{\bar j} h_{i\bar s}\right).
\end{equation*}
\end{lma}
\begin{proof}
The computation is standard, we include here for reader's convenience. 
\begin{equation}
\begin{split}
\partial_t R_{k\bar li\bar j}
&=\partial_t \left(R_{k\bar l i}\,^r g_{r\bar j} \right)\\
&=-  h_{r\bar j}R_{k\bar li}\,^r-g_{r\bar j} \partial_{\bar l}\partial_t \Gamma_{ki}^r\\
&=- h_{r\bar j}R_{k\bar li}\,^r+ \nabla_{\bar l}\nabla_{ k}h_{i\bar j}\\
&=\frac{1}{2}\left(\nabla_k\nabla_{\bar l}+\nabla_{\bar l}\nabla_k \right) h_{i\bar j}-\frac{1}{2} \left(R_{k\bar li}\,^r h_{r\bar j}+R_{k\bar l}\,^{\bar s}_{\bar j} h_{i\bar s}\right).
\end{split}
\end{equation}
\end{proof}



\begin{rem}
By tracing the vector bundle component in Lemma \ref{evo-Rm}, we would have 
\begin{equation}
\begin{split}
\partial_t R_{i\bar j}&=\Delta_g h_{i\bar j}+\left(T_{pi}^k\nabla^p h_{k\bar j}+T_{\bar q\bar j}^{\bar l}\nabla^{\bar q} h_{i\bar l}+g^{p\bar q}T_{pi}^k T_{\bar q\bar j}^{\bar l}h_{k\bar l} \right)\\
&\quad +h^{k\bar l}R_{ i\bar jk\bar l}-\frac{1}{2} \left(S_{i}^k h_{k\bar j}+S^{\bar l}_{\bar j} h_{i\bar l} \right).
\end{split}
\end{equation}

Using the trick by Ustinovskiy \cite{Ustinovskiy2019}, the true variational component is at $h*\Rm$ which do not have any sign in general. This also explains why the deformation on $\Ric$ is false unless we have additional information on $\Rm$, see \cite{Lee2019}.
\end{rem}

\section{local deformation of hermitian metrics}

In this section, we discuss the local deformation of Hermitian metrics. To simplify our argument, we will work on the ball of injectivity radius so that the distance function is smooth there although this is not necessary in the Hermitian content.

Let $(M^n,g)$ be a closed Hermitian manifold with complex dimension $n$. For $p\in M$, we will use $B_g(p,r)$ to denote the $g$-geodesic ball of radius $r$ and $A_g(p,r_1,r_0)=B_g(p,r_1)\setminus B_g(p,r_0)$. Note that Hermitian metric $g$ is at the same time a Riemannian metric. For each $p\in M$, we can find $i_0=\textbf{inj}_g(p)$ such that the exponential map $\exp|_p: B_{\mathbb{R}^{2n}}(0,i_0)\rightarrow B_g(p,i_0)$ is a diffeomorphism. In particular, the square distance function $d_g(p,\cdot)^2$ is smooth on $B_g(p,i_0)$. Let $r<i_0$ and denote $\rho(x)=r^2-d_g(p,x)^2$. We use $d_{g}(p,\cdot)^2$ instead of $d_g(p,\cdot)$ to avoid the non-smooth issue at $x=p$. 

\begin{prop}\label{local-deform}
There is $\mu(n)\in (0,1/2)$ such that the following holds: Suppose $\tilde g$ is a Hermitian metric on $M$ with
\begin{equation}
\sup_M |\Rm(\tilde g)|+|T(\tilde g)|^2<2.
\end{equation}
And $ g_0$ is another Hermitian metric on $M$ such that $\tr_{\tilde g} R^{(TM,g_0)}\geq0$ (resp. $\leq 0$) on $B_{\tilde g}(p,r)$ where $r<\min\{\mathbf{inj}_{\tilde g}(p),1\}$. Then for any $\e>0$ and $k\in \mathbb{N}$, there is another smooth Hermitian metric $g_1$ conformal to $g_0$ such that 
\begin{enumerate}
\item[(a)] $g_1$ agrees with $g_0$ outside $B_{\tilde g}(p,r)$;
\item[(b)] $||g_1-g_0||_{C^k(M,\tilde g)}<\e$;
\item[(d)] $\tr_{\tilde g} R^{(TM,g_1)}>0$ (resp. $< 0$) on $A_{\tilde g}(p,r,(1-\mu)r)$.
\end{enumerate}
\end{prop}
\begin{proof}
We first prove the non-negative case. Let $f$ be the smooth function on $\mathbb{R}$ given by
\begin{equation}
f(s)=
\left\{
\begin{array}{ll}
e^{-1/s} \quad&\text{for}\; s\geq 0;\\
0 \quad&\text{for}\; s\leq 0.
\end{array}
\right.
\end{equation}

Define $F(x)=f(\rho)$ where $\rho(x)=r^2-d_{\tilde g}(x,p)^2$. Let us collect some useful inequality on $\rho$. Clearly, we have $|\partial \rho|_{\tilde g}^2=4(r^2-\rho)$. For second order, note that Chern connection coincides with the canonical connection in almost Hermitian geometry. Hence we can apply \cite[Theorem 4.2]{Tosatti2007} to show that
\begin{equation}\label{rho-lap}
 \Delta_{\tilde g} \rho_{i\bar j} \geq -C_n.
\end{equation}

Define a one parameter family of smooth Hermitian metrics by $ g=e^{-tF}g_0$ so that $\partial_t  g=-F g=-h$. Conclusion (a) is trivial. Conclusion (b) can be done by choosing $g_1=g(t_{\e,k})$ where $t_{\e,k}$ is sufficiently small. It remains to establish conclusion (c). First, we note that Lemma \ref{evo-Rm} implies that for $B_{i\bar j}=e^{Ft} \tilde g^{k\bar l}R_{k\bar l i\bar j}$, we have
\begin{equation}\label{equ-S-evo}
\begin{split}
\partial_t B_{i\bar j}
&=FB_{i\bar j}+e^{Ft}\tilde g^{k\bar l} \partial_t R_{k\bar li\bar j}\\
&=FB_{i\bar j}+\frac{e^{Ft}}{2}\tilde g^{k\bar l}  \left(\nabla_k\nabla_{\bar l}+\nabla_{\bar l}\nabla_k \right) h_{i\bar j}-\frac{e^{Ft}}{2}\tilde g^{k\bar l}  \left(R_{k\bar li}\,^r h_{r\bar j}+R_{k\bar l}\,^{\bar s}_{\bar j} h_{i\bar s}\right)\\
&=\tilde \Delta F \cdot (g_0)_{i\bar j}.
\end{split}
\end{equation}

We now claim that $B_{i\bar j}>0$ on $A_{\tilde g}(p,r,(1-\mu) r)$  for some $0<\mu(n)<1/2$ and $t>0$. We will specify the choice of $\mu(n)$ later. 

For any $z\in A_{\tilde g}(p,r,(1-\mu) r)$ and $v\in T^{1,0}_z M$ so that $g_0(v,\bar v)=1$. By \eqref{equ-S-evo} and \eqref{rho-lap}, if $t>0$, then 
\begin{equation}
\begin{split}
 B(t)_{v\bar v}-B(0)_{v\bar v}
 &=t\cdot  \tilde \Delta F   \\
 &=t\cdot \left(f''|\partial \rho|_{g_0}^2+f'\Delta_{g_0}\rho\right) \\
 &\geq t\rho^{-4} e^{-1/\rho}\left[4(r^2-\rho)-8\rho(r^2-\rho)-C_n\rho^2 \right]\\
&>0
\end{split}
\end{equation}
provided that $\mu(n)$ is sufficiently small. This completes the proof of the non-negative case. The non-positive case can be proved analogously by considering the one parameter family $g(t)=e^{tF}g_0$ instead.
\end{proof}
\begin{rem}\label{deform-rem}
In application, we will work on $B_{\tilde g}(p,r)$ where $\tr_{\tilde g} R^{(TM,g)}>0$ (resp. $<0$) on $B_{\tilde g}\left(p,(1-\mu )r\right)$ and will choose $\e>0$ small enough so that $\tr_{\tilde g} R^{(TM,g(t))}>0$ (resp. $<0$) on $B_{\tilde g}(p,r)$ after deformation. Here we will require $\e$ to be sufficiently small depending also on the positivity on the smaller ball.
\end{rem}

\section{Deformation on compact manifolds}\label{main}

In this section, we will carry the deformation process on $M$. We will follow and modify the argument in \cite{Ehrlich1976}. We will denote the original Hermitian metric as $g$. Since $M$ is closed, the injectivity radius of $x\in M$ has a uniform positive lower bound. Denote
$$\mathbf{inj}_{\tilde g}(M)=\inf_M\left\{ \mathbf{inj}_{\tilde g}(x)\right\}>0.$$


Now we are ready to prove the main deformation theorem.
\begin{thm}\label{global-deform}
Let $M$ is a compact complex manifold and $g_0,\tilde g$ are two smooth Hermitian metric on $M$ such that $\tr_{\tilde g}R^{(TM,g_0)}$ is quasi-positive (resp. quasi-negative). Then for all $\e>0, k\in \mathbb{N}$, there is another Hermitian metric $\hat g_0$ conformal to $g_0$ such that $\tr_{\tilde g}R^{(TM,\hat g_0)}$ is positive (resp. negative) on $M$ and $||\hat g_0-g_0||_{C^k(M,\tilde g)}<\e$.
\end{thm}
\begin{proof}
We will prove the non-negative case. The non-positive case can be proved using the same argument. 

By rescaling, we may assume 
$$\sup_M |\Rm(\tilde g)|+|T(\tilde g)|^2\leq 1.$$
Let $\mu(n)$ be the constant from Proposition \ref{local-deform}. Suppose $p\in M, 0<r_0<\frac14\min \left\{\mathbf{inj}_{\tilde g}(M),1\right\}$ be such that $\tr_{\tilde g}R^{(TM,g_0)}>0$ on $B_{\tilde g}(p,r_0)$.

\begin{claim}
Let $ r_m=(1+{m\mu }(2-\mu)^{-1})r_0$, then there exists Hermitian metrics $\{g_m\}_{m\in \mathbb{N}}$ conformal to $g_0$ such that 
\begin{enumerate}
\item $\tr_{\tilde g}R^{(TM,g_m)}>0$ on $B_{\tilde g}(p,r_m)$;
\item $\tr_{\tilde g}R^{(TM,g_m)}\geq 0$ on $M$;
\item $ ||g_m- g_0||_{C^k(M,\tilde g)}\leq  \e\sum_{i=0}^m 2^{-i-1}$.
\end{enumerate}
\end{claim}
\begin{proof}
[Proof of claim]
The statement is trivially true when $m=0$. Suppose it is true for some $m\in \mathbb{N}_{\geq 0}$. Let $\{p_l\}_{l=1}^{N_m}\subset B_{\tilde g}(p,r_m-r_0)$ so that 
\begin{equation}
\label{set-in}A_{\tilde g}(p,r_{m+1},r_m)\subset \bigcup_{l=1}^{N_m} B_{\tilde g}\left(p_l, \frac{r_0}{1-\mu}\right).
\end{equation}

Now we work on $B_{\tilde g}(p_1,\frac{r_0}{1-\mu})$. By our choice of $r_0$, we have $\frac{r_0}{1-\mu} < \mathbf{inj}_{\tilde g}(p_1)$ and $\tr_{\tilde g}R^{(TM,g_m)}>0$ on $B_{\tilde g}(p_1,r_0)$. Hence, we may apply Proposition \ref{local-deform} and Remark \ref{deform-rem} on $g_m$ so that there is a Hermitian metric $g_{m,1}$ conformal to $g_m$ in which
\begin{enumerate}
\item[(a)] $\tr_{\tilde g}R^{(TM,g_{m,1})}>0$ on $ B_{\tilde g}\left(p_1,\frac{r_0}{1-\mu}\right)\bigcup B_{\tilde g}(p,r_m)$;
\item[(b)] $\tr_{\tilde g}R^{(TM,g_{m,1})}\geq 0$ on $M$;
\item[(c)] $ ||g_{m,1}-g_0||_{\tilde g,\infty}\leq \e \sum_{i=0}^m2^{-i-1}+\e 2^{-m-2}N_m^{-1}$.
\end{enumerate}

Repeat the argument on each $B_{\tilde g}(p_l,r_0)$, $l\leq N_m$ inductively, we will obtain a sequence of Hermitian metrics $\{g_{m,l}\}_{l=1}^{N_m}$ conformal to $g_m$ so that for each $l\in\{1,...,N_m\}$,
\begin{enumerate}
\item[(a)] $\tr_{\tilde g}R^{(TM,g_{m,l})}>0$ on $\bigcup_{i=1}^{l} B_{\tilde g}(p_i,\frac{r_0}{1-\mu})\bigcup B_{\tilde g}(p,r_m)$;
\item[(b)] $\tr_{\tilde g}R^{(TM,g_{m,l})}\geq 0$ on $M$;
\item[(c)] $ ||g_{m,l}-g||_{C^k(M,\tilde g)}\leq \e \sum_{i=0}^m2^{-i-1}+l\e 2^{-m-2}N_m^{-1}$.
\end{enumerate}

Using \eqref{set-in} and the fact that
$$B_{\tilde g}(p,r_{m+1}) \subset {A_{\tilde g}(p,r_{m+1},r_m)}\cup B_{\tilde g}(p,r_m),$$
we have the desired Hermitian metric if we choose $g_{m+1}=g_{m,N_m}$. This proves the claim.
\end{proof}

Since $r_m\rightarrow +\infty$ and $M$ is compact, the process will be terminated at the $N$-th step. Then $\hat g_0=g_N$ will be the desired Hermitian metric that we are looking for.
\end{proof}

\begin{rem}
It is clear from the proof that an analogous result still hold if $M$ is complete non-compact with bounded Chern curvature, bounded torsion and has a uniform injectivity radius lower bound. 
\end{rem}

\begin{proof}[Proof of Corollary \ref{deform-S}]
This follows from Theorem \ref{deform-Rm} by taking $\tilde g=g_0$ and using the fact that the new metric $g_1$ is conformal to $g_0$.
\end{proof}

\begin{rem}The same method also holds for the Chern-scalar curvature $R=\tr_g S=\tr_g \Ric$. This was shown earlier by Yang in \cite[Lemma 3.2]{Yang2019}.
\end{rem}



\begin{thebibliography}{1000}

\bibitem{Aubin1970}Aubin, T., {\sl M\'etriques riemanniennes et courbure}.  J. Differential Geom 4 (1970), 383--424.


\bibitem{Bando}Bando, S., On the classification of three-dimensional compact Kaehler manifolds of nonnegative bisectional curvature, J. Differential Geom. \textbf{19} (1984), no. 2, 283–297.

\bibitem{BedulliVezzoni2017} Bedulli L.; Vezzoni, L., {\sl A parabolic flow of balanced metrics}. J. Reine Angew. Math. 723 (2017), 79-99.



\bibitem{Campana1992} Campana, F., {\sl Connexit\'e rationnelle des vari\'et\'es de Fano}. Ann. Sci. \'Ecole Norm. Sup. (4),
25(5): 539–545, 1992.

\bibitem{CalamaiZhou2020}Calamai, S; Zou, F., {\sl A note on Chern-Yamabe problem}. Differential Geom. Appl. 69 (2020), 101612, 14 pp. 


\bibitem{DiverioTrapani2016}Diverio, S.; Trapani, S., {\sl Quasi-negative holomorphic sectional curvature and positivity of the canonical bundle}, J. Differential Geom., 111(2):303–314, 2019.

\bibitem{Ehrlich1976}Ehrlich, P., {\sl Metric deformations of curvature}. Geometriae Dedicata 5 (1976), no. 1, 1--23.

\bibitem{Ehrlich1976-2}Ehrlich, P., {\sl Local convex deformations, hermitian metrics, and hermitian connections}. Geometriae Dedicata 5.1 (1976): 27-29.


\bibitem{Gill2011}Gill, M., {\sl Convergence of the parabolic complex Monge-Amp\`ere equation on compact Hermitian manifolds}, Comm. Anal. Geom. 19 (2011), no. 2, 277303.





\bibitem{HeierLuWong2010}Heier, G.; Lu, S.; Wong, B., {\sl On the canonical line bundle and negative holomorphic sectional curvature}, Math. Res. Lett., 17 (2010), no.6, 1101–1110.


\bibitem{HeierLuWong2016}Heier, G.; Lu, S.; Wong, B., {\sl \K manifolds of semi-negative holomorphic sectional curvature}. J. Differential Geom., \textbf{104} (3):419–441, 2016.

\bibitem{HeierLuWongZheng2017}Heier, G.; Lu, S.; Wong, B.; Zheng, F., {\sl Reduction of manifolds with semi-negative holomorphic sectional curvature}, Math. Ann. (2018) 372: 951. https://doi.org/10.1007/s00208-017-1638-8.

\bibitem{KMM1992} Koll\'ar, J.;  Miyaoka, Y.; Mori, S., {\sl Rational connectedness and boundedness of Fano manifolds}. J. Differential Geom., 36(3):765–779, 1992.



\bibitem{Lee2019}Lee, M.-C., {\sl Hermitian manifolds with quasi-negative curvature}. Math. Ann. (2020). https://doi.org/10.1007/s00208-020-01997-4


\bibitem{LeeStreets2019}Lee, M.-C.; Streets, J., {\sl Complex manifolds with negative curvature operator}, Int. Math. Res. Not. IMRN , rnz331, https://doi.org/10.1093/imrn/rnz331

\bibitem{LiYuanZhang} Li, Y.; Yuan, Y.; Zhang, Y., {\sl A new geometric flow over \K manifolds}, to appear in Comm. Anal. Geom.


\bibitem{Matsumura2018-1} Matsumura, S.,{\sl On morphisms of compact \K manifolds with semi-positive holomorphic sectional curvature}. arXiv preprint, arXiv:1809.08859, 2018.

\bibitem{Matsumura2018-2}Matsumura, S., {\sl On projective manifolds with semi-positive holomorphic sectional curvature}. arXiv preprint, arXiv:1811.04182, 2018.

\bibitem{Matsumura2018-3}Matsumura, S., {\sl On the image of MRC fibrations of projective manifolds with semipositive holomorphic sectional curvature}. arXiv preprint, arXiv:1801.09081, 2018.

\bibitem{Mok} Mok, N., {\sl The uniformization theorem for compact \K manifolds of nonnegative holomorphic bisectional Curvature}, J. Differential Geometry, \textbf{27} (1988) 179–214


\bibitem{NiZheng2018}Ni, L.; Zheng, F., {\sl Comparison and vanishing theorems for \K manifolds}. Calc. Var. Partial Differential Equations, 57(6):Paper No. 151, 31, 2018.

\bibitem{NiZheng2018-2}Ni, L.; Zheng, F., {\sl Positivity and kodaira embedding theorem}. arXiv preprint, arXiv:1804.09696, 2018.

\bibitem{PhongPicardZhang2018}Phong, D.-H.; Picard, S.; Zhang, X.-W., {\sl Geometric flows and Strominger systems}, Mathematische Zeitschrift, Vol. 288, (2018), 101-113.

\bibitem{PhongPicardZhang2016}Phong, D.-H.; Picard, S.; Zhang, X.-W., {\sl Anomaly flows}, Comm. Anal. Geom. 26 (2018), no. 4, 955--1008.

\bibitem{PhongPicardZhang2016-2}Phong, D.-H.; Picard, S.; Zhang, X.-W., {\sl The Anomaly flow and the Fu-Yau equation}, Ann. PDE 4 (2018), no. 2, Paper No. 13, 60 pp.



\bibitem{PhongPicardZhang2018-2}Phong, D.-H.; Picard, S.; Zhang, X.-W., {\sl A flow of conformally balanced metrics with K\"ahler fixed points}, Math. Ann. 374 (2019), no. 3-4, 2005--2040. 

\bibitem{StreetsTian2010}Streets, J.; Tian, G., {\sl A parabolic flow of pluriclosed metrics}, International Mathematics Research Notices 2010, no. 16 (2010), 3101-3133.



\bibitem{StreetsTian2011}Streets, J.; Tian, G., {\sl Hermitian curvature flow}. J. Eur. Math. Soc. (JEMS) 13 (2011), no. 3, 601--634. 

\bibitem{Streets}Streets, J., {\sl Pluriclosed Flow and the Geometrization of Complex Surfaces}. In: Chen J., Lu P., Lu Z., Zhang Z. (eds) Geometric Analysis. Progress in Mathematics, vol 333. Birkhäuser, Cham

\bibitem{Tosatti2007}Tosatti, V., {\sl A general Schwarz lemma for almost-Hermitian manifolds}. Comm. Anal. Geom. 15 (2007), no. 5, 1063--1086.

\bibitem{TosattiWeinkove2015}Tosatti, V.; Weinkove, B., {\sl On the evolution of a Hermitian metric by its Chern-Ricci form}, J. Differential Geom. 99 (2015), no. 1, 125--163.



\bibitem{TosattiYang2017} Tosatti, V.; Yang, X.-K., {\sl An extension of a theorem of Wu-Yau}. J. Differential Geom., 107(3):573–579, 2017.



\bibitem{Ustinovskiy2019}Ustinovskiy, Y., {\sl The Hermitian curvature flow on manifolds with non-negative Griffiths curvature}, Amer. J. Math., 141(6):1751–1775. 2019.



\bibitem{WuYau2016} Wu, D.-M.; Yau, S.-T., {\sl Negative Holomorphic curvature and positive canonical bundle}, Invent. Math. 204 (2016), no. 2, 595–604.


\bibitem{WuYau2016-2}Wu, D.-M.; Yau, S.-T., {\sl A remark on our paper “Negative Holomorphic curvature and positive canonical bundle”},  Comm. Anal. Geom., \textbf{24} (2016) no. 4, 901–912.


\bibitem{Yang2019}Yang, X.-K., {\sl Scalar curvature on compact complex manifolds}. Trans. Amer. Math. Soc. 371 (2019), no. 3, 2073--2087. 

\bibitem{Yang2018}Yang, X.-K., {\sl RC-positivity, rational connectedness and Yau’s conjecture}. Camb. J. Math. 6 (2018), 183–212.




\bibitem{Yang2018-2}Yang, X.-K., {\sl RC-positivity and the generalized energy density I: Rigidity}. arXiv preprint arXiv:1810.03276 (2018).

\bibitem{Yang2018-3}Yang, X.-K., {\sl RC-positive metrics on rationally connected manifolds}. arXiv:1807.03510.

\bibitem{Yang2018-4}Yang, X.-K., {\sl RC-positivity, vanishing theorems and rigidity of holomorphic maps}. J. Inst. Math. Jussieu (2018): 1-16.


\bibitem{Yang2020}Yang, X.-K., {\sl Compact \K manifolds with quasi-positive second Chern-Ricci curvature},  arXiv preprint arXiv:2006.13884 (2020).



\bibitem{YangZheng2016} Yang, X.; Zheng, F., {\sl On real bisectional curvature for Hermitian manifolds}, Trans. Amer. Math. Soc. 371, 4 (2019), 2703–2718.



\bibitem{Yau1974}Yau, S.-T. {\sl On the curvature of compact Hermitian manifolds}. Invent. Math. 25 (1974), 213--239. 

\bibitem{Yau1978}Yau, S.-T., {\sl On the Ricci curvature of a compact K\"ahler manifold and the complex Monge-Amp\`ere
equation, I}, Comm. Pure Appl. Math. 31 (1978), 339–411.


\bibitem{Yau1982}Yau, S.-T., {\sl Problem section. In Seminar on Differential Geometry}, volume 102 of Ann. of Math.
Stud., pages 669–706. Princeton Univ. Press, Princeton, N.J., 1982.

Chin. Ann. Math. Ser. B, 39 (2018), no. 4, 755-772.

\end{thebibliography}
\end{document}